\documentclass[12pt]{article}
\usepackage{amsmath}
\usepackage{amsfonts,amssymb,mathrsfs,amscd}
\usepackage{amsthm}

\newtheorem{theorem}{Theorem}
\newtheorem{atheorem}{Theorem}
\newtheorem{lemma}{Lemma}

\theoremstyle{definition}
\newtheorem*{defn}{Definition}
\newtheorem{example}{Example}
\newtheorem*{remark}{Remark}

\newcommand{\RR}{\mathbb{R}}
\newcommand{\ZZ}{\mathbb{Z}}
\newcommand{\eps}{\varepsilon}
\newcommand{\df}{d} 
\newcommand{\cnv}[2]{\mathrm{Cnv}(#1,#2)}

\newcommand{\ba}[1]{\begin{array}{#1}}
\newcommand{\ea}{\end{array}}
\newcommand{\beq}[1]{\begin{equation}\label{#1}}
\newcommand{\eeq}{\end{equation}}

\title{Existence of convolution maximizers in $L_p(\RR^n)$ for kernels from Lorentz spaces}
\author{Sergey Sadov\footnotemark[1]\footnote{E-mail: serge.sadov@gmail.com}}
\date{}

\begin{document}
\maketitle

\begin{abstract}
The paper extends an earlier result of G.V.~Kalachev and the author (Sb. Math. 2019) on the existence of a maximizer of convolution operator acting between two Lebesgue spaces on $\RR^n$ with kernel from some $L_q$, $1<q<\infty$. In view of Lieb's result of 1983 about the
existence of an extremizer for the Hardy-Littlewood-Sobolev inequality it is natural to ask whether a convolution maximizer exists for any kernel from
weak $L_q$. The answer in the negative was given by Lieb in the above citation. In this paper we prove the existence of maximizers for kernels from a slightly more narrow class than weak $L_q$, which contains all Lorentz spaces $L_{q,s}$ with $q\leq s<\infty$.  


\end{abstract}

\section{Introduction}
\label{sec:intro}

In this paper we extend the main result of \cite{KS2019}.
We will mostly follow \cite{KS2019} in notation and terminology. 

Let $\mathcal{S}(\RR^n)$ be the Schwarz space and
let $k\in \mathcal{S'}(\RR^n)$, i.e. $k$ is a tempered distribution in $\RR^n$. We 
write, following
\cite{Stepanov1984},%
\footnote{corresponds to $L_p^r$ in \cite{Hormander1960}}
$k\in \cnv{p}{r}$,  
if the convolution operator
$K_k:\,f\mapsto k*f$ defined initially for $f\in\mathcal{S}(\RR^n)$ extends continuously to the map from $L_p(\RR^n)$ to 
$L_q(\RR^n)$,
that is,
$$
 \|k*f\|_{L_q(\RR^n)}\leq C\|f\|_{L_p(\RR^n)}
$$ 
for any $f\in \mathcal{S}(\RR^n)$ with $C$ independent of $f$.

In the sequel we write $\|f\|_p$ instead of 
$\|f\|_{L_p(\RR^n)}$.

We address the question of the existence of a maximizer for the operator $K_k$.
A maximizer is any function $f$ with $\|f\|_p=1$ such that
$$
 \|k*f\|_{q}=\|K_k\|\;(=\|K_k\|\,\|f\|_p).
$$ 

Convolution operators considered in this paper
will have kernels from $\cnv{p}{r}$ where $1<p<r<\infty$.
(In \cite{KS2019}, the exponent of the target space was denoted $r'$.)

If $k\in L_q$ and the triple $(p,q,r)$ satisfies the Young condition
\beq{Y}
 \frac{1}{p}+\frac{1}{q}=1+\frac{1}{r},
\eeq
then $k\in \cnv{p}{r}$ by Young's convolution inequality. 

A weaker sufficient condition is provided by the Hardy-Littlewood-Sobolev theorem: $k\in \cnv{p}{r}\,$ if 
\eqref{Y} holds and $k\in L_{q,\infty}$, the weak $L_q$ space,
see e.g.\ \cite[Theorem~1.4.24]{GrafakosCFA}.%
(We recall the relevant definitions in Sec.~\ref{sec:results}).
The norm of the operator $K_k$ is estimated by the weak $q$-norm of the kernel $k$; this important fact 
follows from the Hardy-Littlewood-Sobolev inequality (see e.g. \cite[Theorem~4.5.3]{HormanderPDE1} or \cite[Theorem~4.3]{LiebLoss}) and the Riesz rearrangement inequality for convolutions \cite[Theorem~3.7]{LiebLoss}. 

The main result of \cite{KS2019}, sharpening the earlier result of Pearson \cite{Pearson1999} (by removing extraneous assumptions), is

\begin{atheorem}
\label{athm:KS}
If $p,q,r$ are related by \eqref{Y}, $1<p<r<\infty$,
 and $k\in L_q$,
then the operator $K_k:\;L_p\to L_r$ has a maximizer.
\end{atheorem}

The ``boundary'' cases $p=1$, $r=\infty$, and $p=r$
are discussed in \cite{KS2019}, too. The full analysis
of the case $p=q=r=1$ is given in \cite{KS2021}.

The present paper is motivated by a desire to connect Theorem~\ref{athm:KS} with another maximizer existence result due to Lieb \cite{Lieb1983}, see also \cite[Sec.~4.8]{LiebLoss}:

\begin{atheorem}
\label{athm:L}
If $p,q,r$ are related by \eqref{Y}, $1<p<r<\infty$,
and
$$
 h(x)=|x|^{-n/q},
$$
then the operator $K_h:\;L_p(\RR^n)\to L_r(\RR^n)$ has a maximizer.
\end{atheorem}

Here $h\notin L_q$, but $h\in L_{q,\infty}$.
One is tempted to conjecture that the condition 
$k\in L_{q,\infty}$ is always sufficient for the existence
of a maximizer. However this is wrong as was pointed out already by Lieb \cite[p.~352]{Lieb1983}. 

We prove the existence result for kernels from a slightly more narrow class $L_{q,\infty,0}\subsetneq L_{q,\infty}$
defined in Sec.~\ref{sec:results}.
It is Theorem~\ref{thm:kweak}.
We obtain it as a consequence of Theorem~\ref{thm:kabstract}, where the assumptions on the kernel have more abstract form and are not very convenient for immediate applications.
In its turn, Theorem~\ref{thm:kweak} yields 
Theorem~\ref{thm:klorentz}, where the sufficient condition is simply $k\in L_{q,s}$, $q\leq s<\infty$, i.e.\ the kernel is assumed to belong to a Lorentz space between the strong $L_q$ and the weak $L_q$. The paper's title refers precisely to the latter result for in that case the class of suitable kernels is most easily understood. 
 
In comparison with \cite{KS2019}, the novelty in this paper is primarily in formulations. For the proof of the base result, Theorem~\ref{thm:kabstract}, we have to make only modest adjustments of some lemmas from \cite{KS2019}.

To be clear, our result does not imply Lieb's. The existence of an extremizer in the Hardy-Littlewood-Sobolev inequality is due to the dilational symmetry of the kernel;
it does sustain truncations of the kernel.

\section{Statement of results}
\label{sec:results}

\begin{theorem}
\label{thm:kabstract}
Let $1<p<r<\infty$ and let $p,q,r$ be related by \eqref{Y}.
Suppose that the distribution $k\in\cnv{p}{r}$ has the following approximation property. For any $\eps>0$,
there exists a measurable function $k_\eps$ such that

\smallskip
(i) $k_\eps$ has finite support;

\smallskip
(ii) $k_\eps\in L_\infty$;

\smallskip
(iii) $\;\|K_{k-k_\eps}\|_{L_p\to L_r}\leq \eps$.

\smallskip
Then the operator $K_k:\;L_p\to L_r$ has a maximizer.
\end{theorem}

Note that (i) and (ii) imply $k_\eps\in L_q$, hence
$k_\eps\in L_p^r$, so $k-k_\eps\in L_p^r$ and the condition
(iii) is meaningful.

Next we exhibit some regular class of measurable functions $k\in L_p^r$ for which the operator $K_k$ has a maximizer. The word {\em regular}%
\footnote{by analogy with definition of a regular integral operator}
means that the function $k$ has the described property if and only if the function $|k|$ has has the property.  

Recall the definitions of the weak $L_q$ space $L_{q,\infty}$ and the  Lorentz spaces $L_{q,s}$.

Given a measurable function $f$ defined on $\RR^n$,
its {\em distribution function}\ is
$$
 \df_f(\lambda)=\left|\{x\mid |f(x)|>\lambda \}\right|,
$$ 
where $|\Omega|$ denotes the Lebesgue measure of the set $\Omega$. 

The {\em decreasing rearrangement}\ of $f$ is the function 
$$
 f^*(t)=\inf\{\lambda>0\mid \df_f(\lambda)\leq t\}.
$$
defined on $[0,+\infty)$.

Put
$$
 \|f\|_{q,\infty}=
\sup_{t>0}t^{1/q}\,f^*(t)
$$
and
$$
 \|f\|_{q,s}=
 \left(\int_0^\infty 
\left(t^{1/q}\,f^*(t)\right)^{s}\frac{dt}{t}\right)^{1/s}
$$
if $0<s<\infty$.

The Lorentz space $L_{q,s}$, $0<s\leq\infty$. consists of measurable functions for which $\|f\|_{q,s}<\infty$.

It is known that $s>s'\;\Rightarrow\;L_{q,s}\supset L_{q,s'}$, see e.g. \cite[Sec.~1.4.2]{GrafakosCFA}.

Also, $L_{q,q}=L_q$. The largest space (for the fixed $q$)
$L_{q,\infty}$ is called the {\em weak $L_q$ space}.
We will now define its special subspace.

\begin{defn}
The space $L_{q,\infty,0}$ is the subspace of
$L_{q,\infty}$ that consists of functions $f$ such that
$$
 \lim_{t\to 0^+} t^{1/q} f^*(t)=
\lim_{t\to \infty} t^{1/q} f^*(t)=0.
$$ 
\end{defn}

It is easy to see that the equivalent condition is
$$
 \lim_{\lambda\to 0^+}\lambda^q\,\df_f(\lambda)=
 \lim_{\lambda\to \infty}\lambda^q\,\df_f(\lambda)=0.
$$ 

\begin{theorem}
\label{thm:kweak}
If $p,q,r$ are related by \eqref{Y}, $1<p<r<\infty$,
 and $k\in L_{q,\infty,0}$,
then the operator $K_k:\;L_p\to L_r$ has a maximizer.
\end{theorem}

We will show that the classical Lorentz spaces
$L_{q,s}$, $s<\infty$, are contained in $L_{q,\infty,0}$
and thus obtain 

\begin{theorem}
\label{thm:klorentz}
If $p,q,r$ are related by \eqref{Y}, $1<p<r<\infty$,
 and $k\in L_{q,s}$ where $q\leq s<\infty$,
then the operator $K_k:\;L_p\to L_r$ has a maximizer.
\end{theorem}

\section{Proofs}
\label{sec:proofs}

\begin{proof}[Proof of Theorem~\ref{thm:kabstract}]
The examination of the proof of Theorem~1 in \cite{KS2019}
shows that the assumption $k\in L_q$
is confined to following lemmas: \textbf{3.7} 
\textbf{4.1}, 
and \textbf{4.2}--\textbf{4.3}.
We will state and prove appropriate modifications of those lemmas in sections (a) to (c) below.
As a matter of fact, the proofs will go essentially along the same lines as in \cite{KS2019}. The difference, in essence, can be described as follows. 
Where ``tails'' are cut from the kernel obtain a core which is bounded and has finite support,
Young's inequality was invoked in the preceding proofs to estimate the operator norms of the corresponding ``small perturbations'' on the spot. Presently for the same purpose we can just refer to the condition (iii) of Theorem~1.  

\subsubsection*{(a) Tightness of a maximizing sequence}

Recall the definition of $\delta$-diameter of a function $f\in L_p(\RR^n)$ in the direction $v\in\RR^n$, $\|v\|=1$
(\cite{KS2019}, Definition~2.2):
$$
 D^p_{\delta,v}(f)=\inf_{b>a}\left\{b-a\left|
\; \int_{a<(x,v)<b} |f(x)|^p\,dx\geq \|f\|_p^p-\delta
\right.\right\}.
$$ 

\begin{lemma} {\rm (Substitute for Lemma 3.7 in \cite{KS2019})}
\label{lem:new3-7}

\noindent
Put $N=\|K_k\|_{L_p\to L_r}$. 
Under the assumptions of Theorem~\ref{thm:kabstract}, suppose
$\eps\in (0,N/3)$ is given and
$k_\eps(x)=0$ for $|x|>R$ (such an $R$ exists by condition (i)). 
Let $\eps_1=\eps/N$
and $f\in L_p$ be any $\eps_1$-maximizer of the operator $K_k$, i.e. $\|f\|_p=1$ and $\|K_k f\|_r\geq N(1-\eps_1)$.  
Then for 
$$
\delta=
\frac{6\eps_1}{1-2^{1-r/p}}
$$ 
and any unit vector $v\in\RR^n$
we have
$$
 D^p_{\delta,v}(f)\leq 8R\,\eps_1^{-p/r}.
$$
\end{lemma}

\begin{remark}
The condition (ii) in the formulation of Theorem~\ref{thm:kabstract} is not needed for this Lemma.
\end{remark}

\begin{proof}
We have the decomposition $K_k=A+B$, where $A=K_{k_\eps}$
is the operator of convolution with bounded function supported in the ball $|x|\leq R$, while $\|B\|\leq\eps$ (condition (iii)). 


Since $f$ is an $\eps$-maximizer for $K_k$, we have
$$
 \|Af\|_r\geq \|K_k f\|_r-\|B f\|_r
\geq \|K_k\|(1-\eps_1)-\eps.
$$
On the other hand,  $\|A\|\leq \|K_k\|+\eps$.
Hence
$$
\|Af\|_r\geq \|A\|(1-\eps_2), 
$$
where
$$
 1-\eps_2=\left(1-\eps_1-\frac{\eps}{N}\right)
\left(1+\frac{\eps}{N}\right)^{-1}.
$$
Thus $f$ is an $\eps_2$-maximizer for $A$.

By the choice of $\eps_1$ we obtain
$$
 \eps_2=\frac{3\eps/N}{1+\eps/N}<3\eps_1.
$$

By Lemma~3.6 of \cite{KS2019}, there exist $\delta>0$
and $L>0$ 
such that for any unit vector $v\in\RR^n$
$$
 D^p_{\delta,v}(f)\leq L.
$$
The expressions for $\delta$ and $L$ are available in explicit form. The cited Lemma~3.6 is applied with parameters $\tau=\eps_2$, $\gamma=r/p>1$, $a=R$.
We take 
$$
 \delta=2\frac{\tau}{1-2^{1-\gamma}}=\frac{2\eps_2}{1-2^{1-r/p}};
$$
then the parameter $\kappa$ in Lemma~3.6 is $\kappa=2\tau$
and a suitable bound $L$ for $D^p_{\delta,v}(f)$ can be taken in the form
$$
 L=8a(\kappa-\tau)^{-1/\gamma}=8R\eps_2^{-p/r}.
$$
The same upper bound is valid for $D^p_{\delta',v}$ for
any $\delta'\geq \delta$. In particular, 
we may take
$\delta'=6\eps_1(1-2^{1-r/p})^{-1}$.

Renaming $\delta'$ into $\delta$ and replacing $\eps_2$
in the above expression for $L$ by $\eps_1$,
which only makes the upper bound coarser (larger) 
since $3/(1+\eps_1)>1$,
we come to the result as stated. 
\end{proof}

A sequence of functions $f_j$ with $\|f_j\|_p=1$ is {\em relatively tight} (Definition 2.5 in \cite{KS2019})
if for any $\delta>0$ there holds
$$
 \sup_j \sup_{\|v\|=1} D^p_{\delta,v}(f_j)<\infty.
$$

As a simple consequence of Lemma~\ref{lem:new3-7}, we deduce the analog of Corollary~3.2 of \cite{KS2019}: 

{\em If an operator $K_k\in\cnv{p}{r}$ satisfies the conditions of Theorem~\ref{thm:kabstract}, then any
maximizing sequence for $K_k$ is relatively tight.}

\subsubsection*{(b) Compactness lemma}

This lemma is used to assert that the operator $K_k$
maps a weakly convergent in $L_p$ sequence to a sequence
convergent in $L_r$-norm on any bounded set in $\RR^n$.

\begin{lemma}  {\rm (Substitute for Lemma 4.1 and Corollary~4.1 in \cite{KS2019})}
\label{lem:new4-1}

Under the assumptions of Theorem~\ref{thm:kabstract}, suppose that the sequence $(f_n)$ with $\|f_j\|=1$ weakly converges in $L_p$ to $f$. Then for any function $\chi\in L_{r}\cap L_\infty(\RR^n)$ we have $\|\chi(x)\cdot(K_k f_j-K_k f)\|_r\to 0$. 
\end{lemma}

\begin{proof}
Given $\eps>0$, we want to find $n_0$ such that
$\|\chi(x)\cdot(K_k f_n(x)-K_k f(x))\|_p<\eps$ for all $j\geq j_0$.

Consider the decomposition $k=k_{\eps/3}+(k-k_{\eps/3})$
provided by the assumptions of Theorem~\ref{thm:kabstract}
(with $\eps/3$ in place of $\eps$). Let $K_k=A+B$ be the corresponding decomposition of the operator $K_k$.

Without loss of generality, we may assume that $\|\chi\|_\infty\leq 1$.

The first part of the proof, dealing with operator $A$ (convolution with $k_{\eps/3}$), is the same as in the proof of Lemma~4.1 of \cite{KS2019}.
Since $k_{\eps/3}\in L_1\cap L_\infty\subset L_{p'}$, the sequence $Af_n$ converges pointwise. Moreover,
$\|\chi\cdot Af_j\|_r\leq \|\chi\|_r\cdot\|k_{\eps/3}\|_{p'}
\cdot \|f_j\|_p\|$ is uniformly bounded (w.r.t. $j$),
so by the dominated convergence theorem there exists $j_0$
such that
$$
 \|\chi\cdot A(f_j-f)\|_r<\frac{\eps}{3},
\qquad j\geq j_0.
$$

The final step of the proof differs from that in \cite{KS2019} in that now it refers to
the condition (iii) of Theorem~\ref{thm:kabstract},
that is, $\|B\|\leq \eps/3$. We get 
$$
 \|\chi\cdot B(f_j-f)\|_r<\|B\|(\|f_j\|_p+\|f\|_p)\leq \frac{2\eps}{3}
$$
for any $n$, due to the fact that $\|f\|_p\leq\lim \|f_j\|_p=1$.

We conclude that for $j\geq j_0$
$$
 \|\chi\cdot K_k(f_j-f)\|_r<\eps,
$$
as required.
\end{proof}

\subsubsection*{(c) Preservation of tightness property under the action of $K_k$}

In \cite{KS2019}, Lemma~4.3 is the result that eventually 
applies in the global scheme of the proof. It is derived
as a consequence of Lemma~4.2. Here we will not need an analog of Lemma~4.2 and will derive the required analog
of Lemma~4.3 using that very lemma in the form 
established in \cite{KS2019}. 


\begin{lemma}  {\rm (Substitute for Lemma 4.3 in \cite{KS2019})}
\label{lem:new4-3}
Suppose that $(f_j)$, $\|f_j\|_p=1$, is a tight sequence in $L_p(\RR^n)$.
That is, for any $\delta>0$, there is a cube $Q\in\RR^n$
such $\int_{\RR^n\setminus Q} |f_j|^p \leq \delta$ for all $j$.
If $k\in\cnv{p}{r}$ satisfies the assumptions of Theorem~\ref{thm:kabstract}, then the sequence $g_j=K_k f_j$
is tight in $L_r(\RR^n)$, that is, for any $\delta>0$
there exists a cube $Q$ such that 
$\int_{\RR^n\setminus Q} |g_j|^r \leq \delta$.
\end{lemma}

\begin{proof}
Consider the decomposition $K_k=A+B$ provided by the conditions of Theorem~1, where $A=K_{k_\eps}$ and $\|B\|_{L_p\to L_r}\leq\eps$. The value of $\eps$ will be
specified later.

We have
$
 g_j=A f_j+B f_j,
$
where $\|B f_j\|_r\leq\eps$ and $A$ is a convolution operator with kernel that certainly lies in $L_q$,
so that Lemma~4.3 of \cite{KS2019} is applicable to it.

Now, given $\delta>0$, let us choose $\eps=\frac{1}{2}\delta^{1/r}$.
Put $\delta_1=2^{-r}\delta$. By Lemma~4.3 of \cite{KS2019},
there exists a cube $Q$ such that for all $j$
$$
 \int_{\RR^n\setminus Q} |Af_j|^r \leq \delta_1. 
$$
By Minkowski's inequality,
\begin{align}
\int_{\RR^n\setminus Q} |g_j|^r & \leq \left(\|Af_j\|_{L_r(\RR^n\setminus Q)}
+\|Bf_j\|_{L_r(\RR^n\setminus Q)}
\right)^r
\\
&\leq (\delta_1^{1/r}+\eps)^r
=\delta,
\end{align}
as required.
\end{proof}

We have revised all the lemmas that needed revision.
The structure of the proof of Theorem~1 in \cite{KS2019} and all other details stay as is. Thus Theorem~\ref{thm:kabstract} of this paper is proved.
\end{proof}

\begin{proof}[Proof of Theorem~\ref{thm:kweak}]
Given $\eps>0$, we are going to exhibit the function
$k_\eps$ satisfying the conditions (i)--(iii) of Theorem~\ref{thm:kabstract}.

By definition of the class $L_{q,\infty,0}$, there exists
$M>0$ such that 
$$
 \lambda\,\df_k(\lambda)<\eps
$$
for all $\lambda>M$. 

Therefore the function
$$
 u(x)=\begin{cases} 
   k(x)\; \text{if}\; |k(x)|>M,
   \\
   0\; \text{if}\; |k(x)|\leq M
\end{cases}
$$
satisfies the inequality
$$
 \lambda\,\df_u(\lambda)<\eps
$$
for all $\lambda>0$. 

Put $v=k-u$.
Then $|v|_\infty\leq M$. 

Also, $d_v(\lambda)\leq d_k(\lambda)$. Hence, by definition
of the class $L_{q,\infty,0}$, there exists
$\delta>0$ such that 
$$
 \lambda\,\df_w(\lambda)<\eps
$$
for all $0<\lambda<\delta$. 

Therefore the function
$$
 w(x)=\begin{cases} 
   u(x)\; \text{if}\; |u(x)|<\delta,
   \\
   0\; \text{if}\; |u(x)|\geq \delta
\end{cases}
$$
satisfies the inequality
$$
 \lambda\,\df_w(\lambda)<\eps
$$
for all $\lambda>0$. 

By the Young-like form of the Hardy-Littlewood-Sobolev inequality we have
$$
\|K_u\|_{L_p\to L_r}\leq C\eps 
$$
and
$$
\|K_w\|_{L_p\to L_r}\leq C\eps, 
$$
where $C$ depends only on $n$ (the dimension of the space), $p$ and $r$. 

The function $y(x)=v(x)-w(x)$ is bounded: $\|y\|_\infty\leq M$ and has support of finite measure:
$\df_y(0)\leq \df_v(\delta)\leq \|k\|_{q,\infty}\delta^{-q}$. 
Therefore there exists $R>0$ such that
$$
 \int_{|x|>R} |y|^q\,dx <\eps.
$$
Put 
$$ 
 z(x)=\begin{cases} 
   y(x)\; \text{if}\; |x|>R,
   \\
   0\; \text{if}\; |x|\leq R.
\end{cases}
$$
By Young's inequality, $\|K_z\|_{L_p\to L_r}\leq \eps$.

The function 
$$
 \tilde k=y-z=k-(v+w+z)
$$
is bounded, has finite support, and
$$
 K_{k-\tilde k}\leq (2C+1)\eps.
$$
Re-denoting $(2C+1)\eps$ into $\eps$, we obtain $\tilde k=k_\eps$ with all the required properties.
\end{proof}

\begin{proof}[Proof of Theorem~\ref{thm:klorentz}]
We will show that if $s<\infty$, then
$L_{q,s}\subset L_{q,\infty,0}$, hence Theorem~\ref{thm:klorentz} will follow from Theorem~\ref{thm:kweak}.

We assume that $f\in L_{q,s}$, that is, $\int_0^\infty(f^*(t))^s\,t^{s/q-1}\,dt<\infty$.
Given $\eps>0$, there exists $T_\eps>0$ such that
the intergral from $T_\eps$ to $\infty$ is less than $\eps$. 
Suppose that $T\geq 2T_\eps$. Then
$$
 \eps>\int_{T_\eps}^T\left(f^*(t)\right)^s\,t^{s/q-1}\,dt
\geq (f^*(T))^s \,\int_{T_\eps}^T t^{s/q-1}\,dt
\geq C\left(f^*(T) T^{1/q} \right)^s,
$$
where
$$
 C=\frac{q}{s} \left(1-2^{-s/q}\right).
$$
Hence $\limsup_{t\to\infty} f^*(t) t^{1/q} \leq (\eps/ C)^{-1/s}$. Since $\eps$ is arbitrary, we obtain
$\lim_{t\to\infty}  f^*(t) t^{1/q}=0.$

We use a similar (actually simpler) argument to prove that $\lim_{t\to 0^+}  f^*(t) t^{1/q}=0$.
Given $\eps>0$, there exists $T_\eps>0$ such that
$
 \int_0^{T_\eps}(f^*(t))^s\,t^{s/q-1}\,dt<\eps.
$
For any $T\in (0,T_\eps)$ we have
$$
 \eps>\int_0^T\left(f^*(t)\right)^s\,t^{s/q-1}\,dt
\geq (f^*(T))^s \,\int_0^T t^{s/q-1}\,dt
\geq \frac{q}{s}\left(f^*(T) T^{1/q} \right)^s,
$$
whence the desired conclusion follows.
\end{proof}

\end{document}